\documentclass[11pt]{amsart}
\usepackage{amsfonts,amssymb,amsmath,amsthm}
\usepackage{fullpage}

\usepackage{amssymb}

\def\R{\mathbb{R}}
\def\Z{\mathbb{Z}}
\def\MM{M}

\def\epsilon{\varepsilon}
\def\ep{\varepsilon}

\newcommand{\commentout}[1]{}

\def\flux{\operatorname{flux}}
\def\pd{\partial}

\newcommand{\norm}[1]{\lVert #1 \rVert}
\newcommand{\br}{\begin{eqnarray}}
\newcommand{\er}{\end{eqnarray}}
\newcommand{\be}{\begin{equation}}
\newcommand{\ee}{\end{equation}}
\newcommand{\baa}{\begin{array}}
\newcommand{\eaa}{\end{array}}
\newcommand{\ba}{\begin{eqnarray}}
\newcommand{\ea}{\end{eqnarray}}

\newtheorem{theorem}{\bf Theorem}[section]

\newtheorem{thm}[theorem]{Theorem}
\newtheorem{lem}[theorem]{Lemma}
 
\newtheorem{cor}[theorem]{Corollary} 

\theoremstyle{definition}
\newtheorem{defin}[theorem]{Definition}
\newtheorem{rmk}[theorem]{Remark}

\begin{document}

\title{A survival guide for feeble fish} 

\author{Dmitri Burago}                                                          
\address{Dmitri Burago: The Pennsylvania State University,                          
Department of Mathematics, University Park, PA 16802, USA}                      
\email{burago@math.psu.edu}                                                     
                                                                                
\author{Sergei Ivanov}
\address{Sergei Ivanov:
St.\ Petersburg Department of Steklov Mathematical Institute,
Russian Academy of Sciences,
Fontanka 27, St.Petersburg 191023, Russia}
\email{svivanov@pdmi.ras.ru}

\author{Alexei Novikov}
\address{Alexei Novikov: The Pennsylvania State University,                          
Department of Mathematics, University Park, PA 16802, USA}                      
\email{anovikov@math.psu.edu}                                

\thanks{The first author was partially supported
by NSF grant DMS-1205597.
The second author was partially supported by
RFBR grant 14-01-00062.
The third author was partially supported
by NSF grant DMS-1515187.
}




\maketitle

\begin{center}
{\it Dedicated to Yu.D.Burago  on the occasion of his 80th birthday.}
\end{center}

\begin{abstract}
As avid anglers we were always interested in the survival chances of fish in turbulent oceans. 
This paper addresses this question mathematically.
We show that a fish with bounded aquatic locomotion speed 
can reach any point in the ocean
if the fluid velocity is incompressible, bounded, and has small mean drift.
\end{abstract}

\section{Introduction}

Suppose  a locally Lipschitz vector field  $V(x)$ and a set  of bounded controls 
\[
\mathcal{A}_t = \{ \alpha \in L^\infty([0,t];\R^d) \;:\; \| \alpha \|_\infty \leq 1 \}
\]
are given. For each $x \in \R^d$ and a control $\alpha \in \mathcal{A}_t$ 
the function $X^{\alpha}_x(s)$ is defined as the unique solution of
\be
\frac{d}{ds} X^{\alpha}_x(s) = V(X^{\alpha}_x(s)) 
+ \alpha(s), \quad s \in [0,t], \quad \quad X^{\alpha}_x(0) = x. \label{controlode1}
\ee
The 
 {\it travel time} from  $x \in \R^d$ to  $y \in \R^d$ is how long it takes to reach $y$ from $x$ with optimal control: 
\be
\tau(x,y) = \inf \{ t \geq 0\;|\; X^\alpha_x(t) = y, \;\;\;\text{for some}\; \alpha \in \mathcal{A}_t \}. \label{taudef1}
\ee
If $\tau(x,y)$ is finite for any $x,y \in \R^d$, then $\tau(x,y)$
can be viewed as a ``non-symmetric metric'' on $\R^d$. 
For us the value of the control $\alpha(s)$ is precisely the strength of 
aquatic locomotion of fish at time $s$, and $V(x)$ is the velocity field of the ambient ocean. 
The finiteness of travel time between any $x$ and $y$ guarantees fish can travel anywhere it wants. Naturally, 
fish can rely on its strength in calm water, that is when $\| V \|_{L^\infty} < 1$, 
but what happens in violent storms? 
Here we intend to clarify the situation in this less obvious regime $\| V \|_{L^\infty}\gg 1$. 

There are two natural obstructions to the finiteness of  $\tau(x,y)$. 
The first one is compressibility. If $V$ has a sink at $x_0$, then, 
depending on the strength of the flow near  $x_0$, we may 
have $\tau(x_0,y) = \infty$  for all $y \neq x_0$. 
The second obstruction is a strong mean drift of $V$.   Indeed, if $V$ is simply a large constant vector, 
then the travel time $\tau(x,y) = \infty$ for many $x$ and $y$. For example, if $V=(v_1, 0, \dots, 0)$, $v_1 \gg 1$,
 then $\tau(x,y) = \infty$
for points $x$ and $y$ if the first coordinate of $x$ is larger than the first coordinate 
of~$y$.  In order to rule out these two natural obstructions, we assume $V$ is incompressible
\begin{equation}\label{tres}
\operatorname{div} V =0,
\end{equation}
and it has small mean drift: 
\begin{equation}\label{dos}
\lim_{L\to\infty} \sup_{x\in\R^d}\left\| \frac{1}{L^d} \int_{[0,L]^d} V(x+y) \, dy \right\| = 0 .
\end{equation}
We also assume that $V$ is bounded:
\begin{equation}\label{uno}
\norm{V}_{L^\infty(\R^d)}  \le \MM < \infty .
\end{equation}
Our main result is the following:

\begin{theorem}\label{main}
Suppose $V$ is locally Lipschitz, bounded~\eqref{uno}, 
incompressible~\eqref{tres}, and has small mean drift~\eqref{dos}.
Then $\tau(x,y)<\infty$ for all $x,y\in\R^d$.
\end{theorem}

If the vector field is globally Lipschitz, we have the following
estimate on travel times.

\begin{thm}\label{mainl} 
Suppose $V$ is incompressible~\eqref{tres},~has small mean drift~\eqref{dos}, 
and $V \in Lip(K)$ i.e.
\begin{equation}\label{lip}
\| V(x) - V(y) \| \leq K |x-y|, \forall x,~y \in \R^d .
\end{equation}
Then the following travel-time estimate
\be\label{geo}
\tau(x,y) \leq C_1 \| x-y \| +C_2,
\ee
holds with some $C_1$ and $C_2$ that depend on $V$ only.
\end{thm}

In the two-dimensional case the estimate \eqref{geo} has been obtained in~\cite{NN}
under slightly different assumptions on $V$.
This estimate has been used in~\cite{NN} to characterize effective behavior of solutions
of the random two-dimensional G-equation, a certain Hamilton-Jacobi partial differential 
equation that arises in modeling propagation of flame fronts in turbulent 
media~\cite{Pet,Wlms85}. Subsequently effective behavior of solutions of the random G-equation was characterized
in~\cite{CS} for any dimension. The approach in~\cite{CS} is different from~\cite{NN}. 
In particular, it relies on ergodic properties of the flow $V$ instead of its geometry. 

Throughout the rest of the paper we assume that $d\ge 3$. The case of $d=2$ easily follows
by adding an auxiliary coordinate.

\begin{rmk}
We do not know whether the Lipschitz condition \eqref{lip} in Theorem \ref{mainl} is necessary.
More precisely, we do not know if it be replaced by the conditions of local Lipschitz continuity and boundedness \eqref{uno}.
Note that the conditions of Lipschitz continuity \eqref{lip} and small mean drift \eqref{dos}
imply boundedness \eqref{uno}.

One possible approach is to sacrifice a part of the available control and use this part
to push $V$ into some compact class of vector fields.  Then we can use
an argument similar to that in the proof of Theorem \ref{mainl}.
\end{rmk}


\begin{rmk}
One may wonder whether uniform convergence is essential in~\eqref{dos}. 
That is, we ask if Theorem \ref{main} remains valid if we replace \eqref{dos}
by the following assumption:
\begin{equation}\label{dos2}
\lim_{L \to \infty} \left\| \frac{1}{(2L)^d} \int_{[-L,L]^d} V(x+y) \, dy \right\| = 0
\end{equation}
for every $x\in\R^d$.
The answer is no.
A counter-example exists already in two dimensions.
Consider the vector field $V$ on $\R^2$ given by
$$
 V(x_1,x_2) =  10\cdot ( sign(x_2), sign(x_1)) .
$$
It is discontinuous, but this can be fixed by convolution
with a suitable mollifier. The resulting field is bounded,
incompressible, and it satisfies \eqref{dos2}.
On the other hand, the fish cannot leave the region
$\{x_1\ge 1, x_2\ge 1\}$.
\end{rmk}

\begin{rmk}
As could be seen from the proof, the small mean drift assumption \eqref{dos}
in Theorem~\ref{main}
can be replaced by a more technical quantitative assumption. Namely it suffices
to assume there is $L_0>0$ such that
$$
\left\| \frac{1}{L^d} \int_{[0,L]^d} V(x+y) \, dy \right\| \le \ep =\ep(d,\MM)
$$
for all $L>L_0$.
Similarly, the parameters determining the constants in Theorem \ref{mainl} are
the uniform bound
$\MM$, the Lipschitz constant $K$, and the above mentioned $L_0$.
\end{rmk}

\subsection*{Motivating ideas of the proof}
Let us start with some motivations. These are not proofs but they indicate
what brought us to a relatively technical argument presented below. 

A feasible counter-example to Theorem \ref{main} would look like this.
Consider a hypersurface $S\subset\R^d$ dividing the space into two regions
$\mathcal R$ and $\mathcal R'$.
Suppose that for every $x\in S$, the vector $V(x)$ 
is directed inward $\mathcal R$ and its normal component 
with respect to $S$ is greater than~1. Then, if the fish starts
at a point in $\mathcal R$, it can never cross $S$ outwards
and hence cannot leave $\mathcal R$.

In fact, {\em any} counter-example should look like this. 
Indeed, let $\mathcal R$ be the set of all points that our fish
can reach from its initial position. Then a simple technical
argument (see Section \ref{sec:localgeom}) shows that the boundary
$\pd\mathcal R$ is locally the graph of a Lipschitz function.
And since the fish cannot leave $\mathcal R$, the flow at the
boundary is directed inward $\mathcal R$ with the normal
component bounded away from zero.
With a little help of the Geometric Measure Theory, one sees
that Lipschitz surfaces are as good as smooth ones for our purposes.
Thus Theorem \ref{main} is equivalent to non-existence of
a hypersurface $S$ with the above properties.

Now let us consider some simple cases.
First of all, if the reachable set $\mathcal R$
is compact, then the contradiction is obvious.
Since the flow field at the boundary points strictly inwards,
the flux through the boundary is nonzero,
and this contradicts the incompressibility condition.

A more interesting situation to consider is the situation when $\mathcal R$ is a tube
(a neighborhood of a straight line) and there is a parallel tube
with opposite flow to cancel the mean drift.
Since the normal component of the flow field on the boundary
is bounded away from zero, the total flux though the boundary
is unbounded. 
This and incompressibility imply that $V$ is unbounded,
contrary to our assumptions.

Things are however more complicated since {\em a priori} the tubes can branch or
widen or have more complicate structure. The main part of
our strategy is to show that this ``branching'' must be exponential and there is
not enough room for this in the Euclidean space.

\begin{rmk}
Our proof is not constructive. It does not give us an actual trajectory 
that the fish can follow to reach a given point. The optimal trajectory 
can be found by studying how the reachable set evolves in time. This could be done
by solving the  G-equation, 
since a certain spatial level set of its solution at a fixed time $t$ is precisely the boundary of the
set  of all points that our fish
can reach before $t$.
\end{rmk}

Our motivating ideas are formulated in a very informal way so far. Let us proceed with
actual proofs of the theorems.

\section{Local geometry of reachable sets}
\label{sec:localgeom}

Throughout the rest of the paper we assume 
that $V$ is a vector field in $\R^d$ satisfying 
the assumptions of Theorem \ref{main}.
The letters $C$ and $c$ denote various positive constants
depending on~$V$. The same letter $C$ may
denote different constants, even within one formula.
We fix the notation $\MM$ for the bound on $\|V\|$,
see \eqref{uno}.

We denote by $\overline S$ the closure of a set $S\subset\R^d$.
We refer to \cite{Federer} for basic properties of rectifiable sets in $\R^d$.
For a $k$-dimensional rectifiable 
set $S\subset\R^d$, $0\le k\le d$, we denote by $|S|$
its $k$-dimensional volume.
For a $(d-1)$-rectifiable co-oriented set $S$
we denote by $\flux(V,S)$ the flux of $V$ through $S$.
(Recall that co-orientation of a hypersurface is a choice of one of the two normal directions.)

In this section we establish basic properties of the set of points
reachable from a fixed point $x\in\R^d$. For technical reasons,
we prefer to work with open reachable sets, defined as follows.

\begin{defin}\label{def-reachable-set}
For $x\in\R^d$ and $\tau>0$, we denote by $\mathcal R_x^\tau$
the set of points reachable from $x$ in positive time less than $\tau$
using controls strictly smaller than~1. That is,
$$
 \mathcal R_x^\tau = \{ y\in\R^d \mid y = X^t_\alpha(x) \text{ for some  $t\in(0,\tau)$ 
 and $\alpha\in\mathcal A_t$ such that $\|\alpha\|_\infty<1$}  \}
$$
We define the \textit{reachable set} $\mathcal R_x$ of $x$
by $\mathcal R_x = \bigcup_{\tau>0} \mathcal R_x^\tau$.
\end{defin}

Clearly $\mathcal R_x^\tau$ and $\mathcal R_x$ are open sets.
We are going to show that the boundary $\pd\mathcal R_x$ enjoys some regularity properties.


%

\begin{defin} Given a point $y\in\R^d$, a vector $v\in\R^d$, and a parameter
$\lambda\in(0,1)$ we define an open cone
\[
C^{\lambda}_{y}(v) = \{ y + tw \mid t>0, \ w \in  \mathbb{R}^d, \
\|w-v\| < \lambda \}.
\]
\end{defin}

If $\|v\|>\lambda$ then $C^{\lambda}_{y}(v)$ is an open ``round cone'' with its apex at~$y$.
If $\|v\|<\lambda$ then $C^{\lambda}_{y}(v)=\R^d$.
If $\|v\|=\lambda$ then $C^{\lambda}_{y}(v)$ is an open half-space.
For fixed $v$ and $\lambda$, the cones $C^{\lambda}_{y_1}(v)$ and $C^{\lambda}_{y_2}(v)$
are parallel translates of each other.


\begin{lem}\label{conecond}
Let $\mathcal R=\mathcal R_x$ and $\lambda\in(0,1)$. 
Then for every $y_0\in\pd{\mathcal R}$ there exists a neighborhood $U$ of $y_0$
such that for every $y\in\overline{\mathcal R}\cap U$
one has $C^{\lambda}_{y}(v_0)\cap U\subset\mathcal R$ where $v_0=V(y_0)$.
\end{lem}

\begin{proof} 
Let $U\ni y_0$ be a convex neighborhood so small
that $\|V(y)-v_0\|< 1-\lambda$ for all $y\in U$.
First consider $y\in \mathcal R\cap U$.
Starting at $x$ and using controls strictly bounded by~1, 
our fish can reach $y$ and then follow
any path $t\mapsto y+tw$, $\|w-v_0\|<\lambda$, until it leaves $U$.
Hence $C^{\lambda}_y(v_0)\cap U\subset\mathcal R$ 
for any $y\in \mathcal R\cap U$.

If $y\in \pd{\mathcal R}\cap U$, the cone $C^{\lambda}_{y}(v_0)$
is the limit of cones $C^{\lambda}_{y'}(v_0)$, $y'\in \mathcal R\cap U$,
$y'\to y$.
More precisely, for every $z\in C^{\lambda}_{y}(v_0)$
we have $z\in C^{\lambda}_{y'}(v_0)$ for all $y'\in\mathcal R$ sufficiently close to $y$.
Since the desired property is already verified for $y'\in\mathcal R$,
it follows that it holds for $y\in\pd\mathcal R$.
\end{proof}


\begin{lem}\label{lip-boundary}
$\pd\mathcal R_x$ is a locally Lipschitz hypersurface, and 
$\mathcal R_x$ locally lies to one side of $\pd\mathcal R_x$.
\end{lem}

\begin{proof}
Let $\mathcal R=\mathcal R_x$ and $y_0\in\pd\mathcal R$.
Fix $\lambda=\frac12$ and let $U$ be a neighborhood of $y_0$
constructed in Lemma \ref{conecond}.
Since $y_0$ is a boundary point of $\mathcal R$,
Lemma \ref{conecond} implies that $C^\lambda_{y}(v_0)\ne\R^d$
for any~$y$.
Choose a Cartesian coordinate system in $\R^d$ such that
the vector $v_0=V(y_0)$ is nonnegatively proportional to the last coordinate vector.
In these coordinates, every cone $C^\lambda_y(v_0)$ is 
the epigraph of the function $F_y\colon\R^{d-1}\to\R$
given by $F_y(u) = y_n +C\,\|u-u_0\|$
where $C=\sqrt{\|v\|^2-1}$, $y_n$ in the last coordinate of $y$,
and $u_0$ is the projection of $y$ to the first coordinate hyperplane.
This fact and Lemma \ref{conecond} imply that $\pd\mathcal R\cap U$
is the graph of a $C$-Lipschitz function and the set $R\cap U$ 
lies above this graph. The lemma follows.
\end{proof}

Lemma \ref{lip-boundary} implies that
$\pd\mathcal R_x$ is a $(d-1)$-dimensional locally rectifiable set
and it has a tangent hyperplane at almost every point.
We equip $\pd\mathcal R_x$ with a co-orientation determined by the choice 
of the normal pointing inwards $\mathcal R_x$.

\begin{lem}\label{big-flux}
For every measurable set $S\subset\pd\mathcal R_x$,
$\flux(V,S) \ge |S|$.
\end{lem}

\begin{proof}
It suffices to verify that, for every $y_0\in\pd\mathcal R_x$ 
such that $\pd\mathcal R_x$ has a tangent hyperplane at~$y_0$,
we have $\langle v_0,n\rangle\ge 1$
where $v_0=V(y_0)$ and $n$ is the inner normal to $\pd\mathcal R_x$ at~$y_0$.
Suppose the contrary and fix $\lambda$ between $\langle v_0,n\rangle$ and 1.
By Lemma \ref{conecond}, $\mathcal R_x$ contains a set
$C^\lambda_{y_0}(v_0)\cap U$ where $U$ is a neighborhood of~$y_0$.
The cone $C^\lambda_{y_0}(v_0)$ contains the ray $\{y_0+tw \mid t>0 \}$ where
$w=v_0-n\langle v_0,n\rangle$. 
The vector $w$ is orthogonal to $n$ and hence belongs to 
the tangent hyperplane to $\pd\mathcal R_x$ at $y_0$.
Thus the tangent hyperplane 
has a nonempty intersection with $C^\lambda_{y_0}(v_0)$, a contradiction.
\end{proof}

\section{Flux estimates}
\label{sec-flux}

First we show that the average flux trough a large $(d-1)$-dimensional cube is small.
This is the only place in the proof where we use the small mean drift assumption.

\begin{lem}\label{Aepsilon}
For every $\ep>0$ there exists $A_0>0$ such that the following holds.
Let $F$ be a $(d-1)$-dimensional cube with edge length $A>A_0$,
then
\begin{equation}\label{dosbissmall}
| \flux(V,F)| \le \epsilon A^{d-1}.
\end{equation}
\end{lem}

\begin{proof}
We may assume that $F = \{0\} \times [0,A]^{d-1}$.
By the small mean drift property \eqref{dos},
there exists $L_0$ such that
\be\label{Lsmalldrift}
 \left\| \frac{1}{L^d} \int_{[0,L]^d} V(x+y) dy \right\| < \ep/2
\ee
for every $L>L_0$ and all $x\in\R^d$.
If $A>L_0$, then $A=mL$ where $m\in\Z$ and $L_0\le L\le 2L_0$.
Consider the layer
$$
 Q = [0,L]\times [0,A]^{d-1} .
$$
It can be partitioned into cubes with edge length $L$, hence \eqref{Lsmalldrift}
implies that
$$
  \frac{1}{LA^{d-1}} \left\|\int_Q V(x)\,dx \right\| < \ep/2.
$$
The Mean Value Theorem implies that there is $t\in[0,L]$ such that
a similar inequality holds for the 
slice $F_t = \{t\}\times [0,L]^{d-1}$ of $Q$ by the hyperplane $\{x_1=t\}$.
Namely
$$
 \frac{1}{LA^{d-1}} |\flux(V,F_t)| =  
  \frac{1}{A^{d-1}} \left|\int_{F_t} \langle V(x), {\bf e}_1 \rangle \, dx \right| < \ep/2,
$$
where the integration in the right-hand side is taken against 
the $(d-1)$-dimensional volume.
Let $Q_t=[0,t]\times[0,A]^{d-1}$. Incompressibility implies that the flux of $V$
through the boundary $\partial Q_t$ is zero. This boundary contains two ``large''
cubic faces $F$ and $F_t$, and the area of the remaining part is ``small'':
$$
 |\partial Q_t\setminus (F\cup F_t)| = (2d-2) t A^{d-2} \le C L A^{d-2}  .
$$
Hence
$$
 |\flux(V,F)| \le \frac\ep2 A^{d-1} +   C L \MM A^{d-2} . 
$$
Choosing $A$ large enough we obtain \eqref{dosbissmall}.
\end{proof}

Denote by $I_t$ the $n$-dimensional cube with edge length $2t$ centered at zero:
$I_t = [-t,t]^d$.
The following two lemmas concern estimates on flux of $V$ through 
subsets of $\pd I_t$.

\begin{lem}\label{lemma-flux-iso}
Let $D$ be a subset of $\partial I_t$ with a $(d-2)$-rectifiable boundary $\pd D$. 
Then
\be\label{flux-isoperimeter}
 |\flux(V,D)| \le C|\pd D|^{(d-1)/(d-2)}
\ee
where $C=C(d,\MM)$.
\end{lem}

\begin{proof}
By incompressibility,
$$
 |\flux(V,D)| = |\flux(V,\pd I_t\setminus D)| .
$$
Hence
$$
|\flux(V,D)| \le \MM \min\{|D|,|\pd I_t\setminus D|\} .
$$ 
By the isoperimetric inequality,
$$
 |\pd D| \ge C\min\{|D|,|\pd I_t\setminus D|\}^{(d-2)/(d-1)}
$$
for some $C=C(d)$. The last two inequalities imply \eqref{flux-isoperimeter}.
\end{proof}

\begin{lem}\label{lemma-flux-mixed}
For every $\ep>0$ there exist $A_0>0$ and $C_0=C_0(\ep,V)>0$ such that for every $t>A_0$
the following holds.
Let $D$ be a subset of $\partial I_t$ with a $(d-2)$-rectifiable boundary $\pd D$.
Then
\be\label{flux-mixed}
 |\flux(V,D)| \le C_0|\pd D|+\ep t^{d-1} .
\ee
\end{lem}

\begin{proof}
By Lemma \ref{Aepsilon}, there exists $A_0>0$ such that the flux of $V$
through every $(d-1)$-dimensional cube with edge length $A>A_0$ is bounded
by $\frac\ep{2d}A^{d-1}$.
If $t>A_0$, then $t = m A$ where $m\in\Z$ and $A_0\le A\le 2A_0$.
We divide $\pd I_t$ into $(d-1)$-dimensional cubes $Q_i$, $i=1,2,\dots,2dm^{d-1}$,
with edge length $A$.
For each $i$, define $P_i=|\pd D\cap Q_i|$
and
$ S_i = \min\{|Q_i\cap D|, |Q_i\setminus D|\} $.
By the isoperimetric inequality,
$$
 S_i \le C P_i^{(d-1)/(d-2)}
$$
for some $C=C(d)$. And, trivially,
$$
 S_i \le |Q_i| = A^{d-1} .
$$
Combining these two inequalities we obtain
\be\label{CPA}
 S_i = S_i^{(d-2)/(d-1)} S_i^{1/(d-1)} \le C P_i A .
\ee
By the choice of $A_0$ we have
$$
 |\flux(V,Q_i)| \le \frac\ep{2d} |Q_i| .
$$
Hence
$$
 \bigl| |\flux(V,Q_i\cap D)|-|\flux(V,Q_i\setminus D)| \bigr| 
 \le \frac\ep{2d} |Q_i| .
$$
Since at least one of the terms $|\flux(V,Q_i\cap D)|$ and 
$|\flux(V,Q_i\setminus D)|$ is bounded by $\MM S_i$, it follows that
$$
 |\flux(V,Q_i\cap D)| \le \MM S_i + \frac\ep{2d} |Q_i| \le CP_iA + \frac\ep{2d} |Q_i|.
$$
Summing up over all $i$ and setting $C_0=2CA_0$ yields \eqref{flux-mixed}.
\end{proof}

\section{Proof of Theorem \ref{main}}

Theorem \ref{main} is an immediate corollary of the following lemma.

\begin{lem}\label{reachable-set-is-all}
For every $x\in\R^d$, the reachable set $\mathcal R_x$ is the entire $\R^d$.
\end{lem}

\begin{proof}
Arguing by contradiction, assume that $\mathcal R_x\ne\R^d$ for some $x\in\R^d$.
Lemma \ref{lip-boundary} implies that $\pd \overline{\mathcal R_x} = \pd\mathcal R_x$
and hence $\overline{\mathcal R_x}\ne\R^d$.
For $t>0$ denote 
$$
D_t = \overline{\mathcal R_x} \cap \pd I_t, \quad S_t = \pd\mathcal R_x \cap I_t ,
\quad L_t = \pd\mathcal R_x \cap \pd I_t,
$$
where $I_t$ is the cube defined in Section \ref{sec-flux}.
Since $\pd\mathcal R_x$ is a nonempty locally Lipschitz hypersurface,
we have $|S_t|>0$ for all $t>t_0$,
where $t_0$ is the distance from 0 to $\pd\mathcal R_x$
and $|S_t|$ is the $(d-1)$-dimensional volume of $S_t$.

The sets $L_t$ are slices of $\pd\mathcal R_x$ by level sets 
of the 1-Lipschitz function $x\mapsto\max|x_i|$.
Hence $L_t$ is a $(d-2)$-rectifiable set for almost every $t$.
In the sequel we consider only those values $t>t_0$ for
which $L_t$ is $(d-2)$-rectifiable.
In particular, the $(d-1)$-dimensional volume of $L_t$ is zero.
This implies that $L_t=\pd D_t$, where $\pd D_t$ denotes the boundary of $D_t$ in $\pd I_t$.
Let $P(t)=|L_t|=|\pd D_t|$ denote the $(d-2)$-dimensional volume of this set.
By the co-area inequality,
$$
 |S_t| \ge A(t) := \int_0^t P(s) \, ds .
$$

Observe that the union $D_t\cup S_t$ is the boundary of the set $\overline{\mathcal R_x}\cap I_t$.
In addition, $D_t\cap S_t=L_t$ which is $(d-2)$-dimensional. 
Hence, by the incompressibility condition~\eqref{tres},
$$
 |\flux (V, D_t)| = |\flux (V, S_t)| \ge  |S_t| .
$$
where the last inequality follows from Lemma~\ref{big-flux}.
Thus
\be\label{fluxVDt}
 |\flux (V, D_t)| \ge |S_t| \ge A(t) .
\ee
By Lemma \ref{lemma-flux-iso},
$$
 |\flux (V, D_t)|  \le C P(t)^{(d-1)/(d-2)} .
$$
In particular $P(t)>0$ for $t>t_0$. It follows that
$A(t)>0$ for all $t>t_0$ and
$$
 \tfrac{d}{dt}A(t) = P(t) \ge c\, A(t)^{(d-2)/(d-1)} .
$$
for some $c>0$ and almost every $t>t_0$. Therefore
$A(t) \ge c_0 (t-t_0)^{d-1}$ and hence
$$
 |\flux (V, D_t)| \ge c_0 (t-t_0)^{d-1}
$$
where $c_0>0$ is some constant depending on $V$.
Now we apply Lemma \ref{lemma-flux-mixed} to $\ep=c_0/3$ and obtain
\be\label{two-polynomials}
 c_0(t-t_0)^{d-1} \le |\flux (V, D_t)| \le C P(t) + \frac{c_0}3 t^{d-1} .
\ee
For all sufficiently large $t$ we have
$(t-t_0)^{d-1} > \frac23t^{d-1}$ and
therefore
$
C P(t) + \frac{c_0}3 t^{d-1} \ge  \frac{2c_0}3t^{d-1}
$.
Hence $ \frac{c_0}3 t^{d-1} \le C P(t)$.
Thus \eqref{two-polynomials}
implies
$$
 |\flux (V, D_t)| \le C P(t) 
$$
for all sufficiently large $t$.
Combining this with \eqref{fluxVDt} yields
$
 \tfrac d{dt} A(t) \ge c A(t)
$
for all sufficiently large $t$.
Hence $A(t)$ grows exponentially.

On the other hand,
$$
 |\flux(V,D_t)| \le \MM |\pd I_t| \le C t^{d-1} .
$$
This and \eqref{fluxVDt} imply that $A(t)$ grows 
at most polynomially.
This contradiction proves Lemma \ref{reachable-set-is-all}
and Theorem \ref{main}.
\end{proof}

\begin{cor}\label{add}
For every compact set $B\subset\R^d$ there exists $\tau_0=\tau_0(V,B)>0$
such that $\tau(x,y)\le\tau_0$ for all $x,y\in B$.
\end{cor}

\begin{proof}
Fix $x_0\in B$.
Recall that $\mathcal R_{x_0}$ is the union of nested open sets $\mathcal R_{x_0}^\tau$, $\tau>0$.
Since $\mathcal R_{x_0}=\R^d$ and $B$ is compact, it follows that $B\subset\mathcal R_{x_0}^{\tau_1}$
for some $\tau_1>0$. Thus $\tau(x_0,y)\le \tau_1$ for all $y\in B$.

To complete the proof, we show that $x_0$ can be reached from any point $x\in B$ 
in a uniformly bounded time.
This is equivalent to reaching $x$ from $x_0$
in the flow defined by the opposite vector field $-V$.
Applying the above argument to $-V$ yields 
that there is $\tau_2>0$ such that $\tau(x,x_0)\le\tau_2$ for all $x\in B$.
Hence $\tau(x,y)\le\tau_0:=\tau_1+\tau_2$ for all $x,y\in B$.
\end{proof}

\section{Proof of Theorem~\ref{mainl}}
\label{sec:proof2}

Now we assume that $V$ satisfies the assumptions of Theorem \ref{mainl}.
To prove Theorem \ref{mainl} it suffices to verify the following:
there is a constant $C>0$ such that $\tau(x,y)\le C$
for all $x,y\in\R^d$ satisfying $\|x-y\|\le 1$.

Suppose the contrary.
Then there exist two sequences of points $\{x_n\}$ and $\{y_n\}$ in $\R^d$
such that $\|x_n-y_n\|\le 1$ and $\tau(x_n,y_n)>n$ for all $n$.
Consider shifted vector fields $V_n$ given by 
$$
V_n(x)=V(x-x_n) .
$$
By the Arzela--Ascoli Theorem, there exists a subsequence $\{V_{n_i}\}$
which converges to a vector field $V_0$ uniformly on compact sets.
The vector field $V_0$ inherits boundedness, Lipschitz continuity,
and the small mean drift property \eqref{dos} from $V$.

Thus we can apply Theorem \ref{main} and Corollary \ref{add}
to $V_0$ in place of $V$. By means of rescaling, this works
even if the control is $\frac12$-bounded, that is
if we consider travel times
\be\label{halftau}
  \widetilde\tau(x,y) = \inf \{t\ge 0 \mid X_x^{\alpha/2}(t)=y \text{ for some $\alpha\in\mathcal A_t$} \} .
\ee
By Corollary \ref{add} there is $\tau_0$ such that $\widetilde\tau(x,y)\le\tau_0$ for all $x,y$
from the unit ball centered at 0,
where $\widetilde\tau$ is defined by \eqref{halftau} for the vector field $V_0$.
In particular the travel times $\widetilde\tau(0,y_n-x_n)$ are bounded by $\tau_0$.
Since our vector field is bounded, all trajectories realizing these travel times
are confined to some ball $B_R(0)$. For $n_i$ large enough we have 
$\|V_{n_i}-V_0\|<\frac12$ on $B_R(0)$ and hence all such trajectories
are also trajectories for $V_{n_i}$ with 1-bounded control.
Thus $\tau(x_n,y_n)\le\widetilde\tau(0,y_n-x_n)\le\tau_0$, a contradiction.


\begin{thebibliography}{abc99xyz}



\bibitem{CS} P. Cardaliaguet, P.E. Souganidis, 
{\em Homogenization and enhancement of the G-equation in random environments}, 
Comm. Pure Appl. Math. 66 (10) (2013) 1582--1628.




\bibitem{Federer}
H. Federer,
Geometric measure theory, Springer-Verlag, 1969.


\bibitem{NN} J. Nolen, A. Novikov, 
{\em Homogenization of the G-equation with incompressible random drift}. Commun. Math. Sci. 9 (2011), no. 2, 561--582.



\bibitem{Pet} N. Peters, Turbulent Combustion, Cambridge University Press, Cambridge, 2000.



\bibitem{Wlms85} F.A. Williams,
{\em Turbulent Combustion}, in The Mathematics of Combustion, J.D. Buckmaster, 
Ed. Society for Industrial and Applied Mathematics, 1985, pp. 97--131.



\end{thebibliography}
\end{document}